\documentclass[leqno,12pt]{article}
\usepackage{amsmath,amsfonts,amssymb,amsthm}
\usepackage[]{graphicx}
\newtheorem{thm}{Theorem}
\newtheorem{cor}{Corollary}
\newtheorem{prop}{Proposition}
\newtheorem*{thma}{Theorem A}
\theoremstyle{definition}

\date{}

\newcommand \Fc{\mathcal F}
\newcommand \Fcc{\mathcal F^2}
\newcommand \Ic{\mathcal I}
\newcommand{\parcv}[2]{\frac{\partial#1}{\partial#2}}
\newcommand{\pprc}[5]{\frac{\partial^2#1}{\partial#2^#3\partial#4^#5}}
\newcommand{\prc}[3]{\frac{\partial#1}{\partial#2^#3}}
\newcommand\RR{\mathbb R}
\newcommand{\en}{_{i=1}^n}
\newcommand\pa{\partial}
\newcommand\paa{\partial^2}
\newcommand\rank{\operatorname{rank}}

\begin{document}
\renewcommand*{\thefootnote}{\fnsymbol{footnote}}
\begin{center}
{\bf\Large ON PROJECTIVELY FLAT FINSLER SPACES}
\bigskip

T.~Q.~BINH$^1$ , D.~CS.~KERT\'ESZ$^{2,}$\footnote{The publication is supported by the T\'AMOP-4.2.2/B-10/1-2010-0024 project. The project is co-financed by the European Union and the European Social Fund.} and L.~TAM\'ASSY$^3$
\bigskip

{\small $^1$University of Debrecen, Institute of Mathematics,\\
H–-4010 Debrecen, P.O. Box 12, Hungary\\
e-mail: binh@science.unideb.hu\\
\bigskip

$^2$University of Debrecen, Institute of Mathematics,\\
H–-4010 Debrecen, P.O. Box 12, Hungary\\
e-mail: kerteszd@science.unideb.hu\\
\bigskip

$^3$University of Debrecen, Institute of Mathematics,\\
H–-4010 Debrecen, P.O. Box 12, Hungary\\
e-mail: tamassy@science.unideb.hu}
\end{center}

{\bf Abstract.}\quad First we present a short overview of the long history of projectively flat Finsler spaces. We give a simple and quite elementary proof of the already known condition for the projective flatness,
and we give a criterion for the projective flatness of the space (Theorem~\ref{thm:1}). After this we obtain a second-order PDE system, whose solvability is necessary and sufficient for a Finsler space to be projectively flat (Theorem~\ref{thm:2}). We also derive a condition in order that an infinitesimal transformation takes geodesics of a Finsler space into geodesics. This yields a Killing type vector field (Theorem~\ref{thm:3}). In the last section we present a characterization of the Finsler spaces which are projectively flat in a parameter-preserving manner (Theorem~\ref{thm:4}), and we show that these spaces over $\RR^n$ are exactly the Minkowski spaces (Theorem~\ref{thm:5} and \ref{thm:6}).

\newpage

\section{Introduction and historical overview}

Let $F^n=(M,\Fc)$ be a Finsler space over a connected $n$-dimensional base manifold $M$ with a regular, $1^+$-homogeneous and strongly convex Finsler metric (fundamental function, Finsler function) $\Fc$ \cite{BCS}. $F^n$ or $\Fc$ is said to be \emph{locally projectively flat} (\emph{locally projectively Euclidean} or, by M. Matsumoto's terminology, a \emph{Finsler space with rectilinear  extremals}) if there exists an atlas $(U_\alpha,\varphi_\alpha)_{\alpha\in \mathcal A}$ of $M$ satisfying the following condition: for any geodesic
$\gamma\colon I\to U_\alpha$
of $F^n$ there exists a strictly monotone smooth function $f\colon I\to \RR$, and two vectors $a,b\in \RR^n$, $a\ne 0$ such that
\begin{equation}
\varphi_\alpha(\gamma (t))=f(t)a+b,\quad t\in I. \tag{1a}\label{eq:1a}
\end{equation}
In this case the curve $\varphi_\alpha\circ\gamma$ is a straight line in the Euclidean space $\RR^n$, but it is not necessarily a geodesic, since its speed is not constant in general.

In this paper, our investigations are of purely local nature, so we drop the word `locally' and we say simply \emph{projectively flat}.
If, in particular, for any geodesic $\gamma$, we have $f(t)=t$ for all $t\in I$, i.e., the image curve $\varphi_\alpha\circ \gamma$ of $\gamma$ is the affine parametrization
\begin{equation}
t\in I \mapsto ta+b\in \varphi_\alpha(U_\alpha) \tag{1b} \label{eq:1b} \end{equation}
of a straight line segment in $\varphi_\alpha(U_\alpha)\subset \RR^n$, then we say that $F^n$ is projectively flat in a \emph{parameter-preserving manner}. In this case the image curve does remain a geodesic.
Previous investigations concerned mainly with the general case \eqref{eq:1a}, and less attention was paid to the special case \eqref{eq:1b}.

Although the question to find Lagrangians whose geodesics are straight lines  emerged already in the 19th century (raised by G.~Darboux), the real starting point of the investigations of projectively flat metrics is Hilbert's fourth problem \cite{Hi} raised on the International Congress of Mathematicians (Paris 1900), in which he asked about the spaces in which the shortest curves between any pair of points are straight line segments. The first answer was given by Hilbert's student G.~Hamel. In a 34 pages long paper he found necessary and sufficient conditions in order that a space satisfying an axiom system, which is a modification of Hilbert's system of axioms for Euclidean geometry, removing a strong congruence axiom and including the Desargues axiom, be projectively flat (\cite[esp.\ p.~250 (I,b)]{Ha}; see also \cite[p.~185]{S1}; \cite[p.~66]{B1}; or \cite{C}). Later E.~Bompiani (1924) showed that among the Riemannian manifolds exactly that of the constant curvature  are projectively flat.

The term `projective metric' was introduced by Busemann and Kelly \cite{BK}. Following them, a metric, i.e., a distance function $d\colon\RR^n\times \RR^n\to\RR^+$ is called to be projective if $d$  is continuous with respect to the canonical topology of $\RR^n$, and $d$ is `additive along the straight lines', i.e., $d(a,b)+d(b,c)=d(a,c)$, whenever $a$, $b$, $c$ are points on a straight line in the given order. In this more general context the `fourth problem requests, among many other  things, a method of construction for the cone of such projective metrics, and this task lies at the heart of the problem' (quoted verbatim from R. Alexander \cite{A}).

The first 75 years of study related to the question is summarized in a survey article of Busemann \cite{B2}. Obviously, the three classical geometries (Euclidean, hyperbolic, and elliptic) solve the problem over a suitable domain of $E^n$. Hilbert himself mentioned two solutions of non-Riemannian type: the Minkowski spaces (i.e., the finite dimensional Banach spaces), and a modification of the Cayley-Klein construction of hyperbolic geometry, now called Hilbert geometry. An interesting, non-symmetric projective metric was discovered by Funk in 1929 \cite{F}.

Differentiable distance functions supplemented with certain mild and natural additional conditions yield also Finsler metrics \cite{T}. It is well-known that a Finsler manifold with Hilbert, Funk or Klein metric is projectively flat (\cite[pp.~32--33]{S1}; \cite[pp.~105--106]{B1}).

\setcounter{equation}{1}

For general Finsler spaces A. Rapcs\'ak gave conditions for projective flatness (\cite{R}; or \cite[section 12.2]{S1}; see also \cite{C}, and \cite[Th.8.1]{Szi}). Both Hamel's and Rapcs\'ak's results say basically that the condition for the projective flatness is the existence of a local coordinate system $(x)$ on the base manifold in which the Finsler function satisfies the system of partial differential equations
\begin{equation}
\frac{\partial^2\Fc}{\partial x^k\partial y^i} y^k=\frac{\partial^2\Fc}{\partial x^i \partial y^k}y^k. \tag*{$(*)$} \label{eq:cs}\end{equation}
Actually, Rapcs\'ak gave conditions in order that two Finsler spaces $(M,\Fc)$ and $(M,\bar\Fc)$ admit projective mapping on each other. This condition is
\begin{equation}
\Fc_{|i}=\frac{\partial \Fc_{|k}}{\partial y^i} y^k, \label{eq:2}
\end{equation}
where $|$ means the horizontal Berwald derivative determined by $\bar \Fc$. If, in particular, $(M,\bar\Fc)$ is a Euclidean space in a Cartesian coordinate system, then \eqref{eq:2} reduces to \ref{eq:cs}.

For a recent account in the different formulations and coordinate free reformulation of Rapcs\'ak's equations we refer to \cite{Szi} and \cite{BSz}. In two dimensions, using an integral representation of the Finsler function, A.~V.~Pogorelov presented an elegant solution of Hamel's equation \cite{P}. He also showed that the smoothness of the Finsler function is not an essential condition, since any continuous projective Finsler function can be uniformly approximated by smooth projective Finsler functions solving Hilbert's problem on each compact subset. He gave a comprehensive solution of Hilbert's problem in~2 and~3 dimensions. However some important questions remained open. Among others, the following: how can the continuous Finsler functions be constructed in dimensions greater than two? The exciting questions arising from Pogorelov's investigations were discussed and exhaustively answered by Z.~I.~Szab\'o \cite{Sza}. Many other papers dealt with the problem from different aspects, as that of M. Matsumoto \cite{M}, \cite{MW}; X.~Mo, Z.~Shen and C.~Yang \cite{MSY}; Alvarez Paiva \cite{AP} etc. In the last time especially the projectively flatness of Randers and Einstein spaces came into the lime light. See T.~Q.~Binh and X.~Cheng~\cite{BC}, Z.~Shen, and C.~Yildrim \cite{SY}, B.~Li and Z.~Shen \cite{LS}, Z.~Shen \cite{S1,S2,S3}, X.~Cheng and M.~Li \cite{CL}.

Recently M.~Crampin \cite{C} introduced the notion of pseudo-Finsler spaces by a slight modification of the basic assumptions on the fundamental function. In a pseudo-Finsler space the positiveness of $\Fc$ is dropped, positive homogeneity is kept, and the positive definiteness of the Hessian $\frac{\partial^2\Fc^2}{\partial y^i\partial y^k}$ is replaced by the condition that the `reduced Hessian` of $\Fc$ defines at each point $v\in T_pM\setminus\{0_p\}$ a \emph{non-degenerate} symmetric bilinear form on the factor space $T_pM/\langle v\rangle$. Then
\begin{equation}
\rank \left(\frac{\pa^2\Fc}{\pa y^i\pa y^k}\right)=n-1. \label{eq:4}
\end{equation}
If $\Fc >0$ on $\mathring TM$, then a pseudo-Finsler function becomes a Finsler function (see Crampin \cite[Lemma 2]{C}). He also realized the importance of \eqref{eq:4} in the proof of \ref{eq:cs}.

In this paper first we want to present a simple proof of condition \ref{eq:cs} for the projective flatness. This proof is very near
to the known ones (Crampin, Shen, etc.), but it is quite short, and uses elementary considerations only. Starting with a smooth and $1^+$-homogeneous Lagrange function $L$, from the powerful condition of the non-vanishing of the Gaussian curvature $K$ of the level surfaces of $L$ we conclude relation \eqref{eq:4}. Thus $L$ becomes a Finsler function. Then from \eqref{eq:4} the projective flatness follows easily (Theorem~\ref{thm:1}). After this we obtain a second-order PDE system, whose solvability is necessary and sufficient  for the existence of a coordinate system $(x)$, in which \ref{eq:cs} is satisfied, that is for the projective flatness of $F^n$ (Theorem~\ref{thm:2}). Also in form of a PDE system we find necessary and sufficient conditions for the existence of a vector field $X$ on the base manifold, such that the local one-parameter group generated by $X$ takes Finsler geodesics as point sets into Finsler geodesics in the same sense (Theorem~\ref{thm:3}). Thus, this vector field $X$ is an analogue of a Killing vector field, in the case when the transformations take geodesics into geodesics as point sets. If both the starting and image curves are geodesics (as parameterized curves) then the one-parameter group consists of affine transformations.
The obtained characterization of these spaces is very similar to \ref{eq:cs} (Theorem~\ref{thm:4}). Finally we show that a Finsler function $\Fc$ over $\RR^n$ is projectively flat in a parameter-preserving manner, if and only if $(\RR^n,\Fc)$ is a Minkowski space (Theorem~\ref{thm:5} and \ref{thm:6}).

\section{Projectively flat Finsler spaces}

As we have mentioned in the introduction, our considerations are purely local. So we assume in most occasions, that our base manifold $M$ admits a global chart $u=(u^1,\dots,u^n)$. Then on the total manifold of the tangent bundle $\tau\colon TM\to M$, we have the induced global chart $(x,y)=((x^i)_{i=1}^n,(y^i)_{i=1}^n)$, where
\[
x^i:=u^i\circ\tau,\quad y^i(v):=v(u^i);\quad v\in\ TM,\ i\in\{1,\dots,n\}.
\]
The coordinate vector fields $\parcv{}{u^i}$ form a (global) frame field of $TM$, while $\left(\parcv{}{x^i},\parcv{}{y^i}\right)_{i=1}^n$ is a frame field of $TTM$. We use the notation $f^{\mathsf v}:=f\circ\tau$, for the vertical lift of a function $f$ on $M$.

{\bf 1.} We summarize some basic facts concerning Finsler functions and their geodesics. Let $\gamma\colon[t_1,t_2]\to M$ be a curve of class $C^2$ in $M$ with component functions $\gamma^i:=u^i\circ\gamma$. Then
\begin{equation}\label{eq:crvcoord}
\dot\gamma=(\gamma^i)'\left(\prc{}ui\circ\gamma\right)\mbox{ and } \ddot\gamma=(\gamma^i)'\left(\prc{}xi\circ\dot\gamma\right)+(\gamma^i)''\left(\prc{}yi\circ\dot\gamma\right).
\end{equation}
are the velocity and the acceleration vector fields of $\gamma$, respectively. We have the obvious relations
\begin{equation}\label{eq:crvcomp}
  x^i\circ\dot\gamma=\gamma^i,\quad y^i\circ\dot\gamma=(\gamma^i)'.
\end{equation}

Now suppose that a Finsler function $\Fc\colon TM\to\RR$ is specified for $M$. The \emph{Finsler length} of a $C^1$ curve $\gamma\colon[t_1,t_2]\to M$ is
\[
\mathbb F(\gamma):=\int_{t_1}^{t_2} F\circ\dot\gamma=\int_{t_1}^{t_2} F(\dot\gamma(t))\,dt,
\]
where, owing to the positive homogeneity of $\Fc$, the integral is invariant under positive reparametrization. By a \emph{geodesic} of $\Fc$ we mean an extremal of the functional $\mathbb F$ which has positive constant speed. A curve $\gamma\colon[t_1,t_2]\to M$ of class $C^2$ is an extremal of $\mathbb F$ if, and only if, it is a solution of the Euler\,--\,Lagrange equations $\prc\Fc xi-\frac d{dt}\prc\Fc yi=0$, i.e.,
\begin{equation}\label{eq:Eul}
  \prc\Fc xi\circ\dot\gamma-\left(\prc\Fc yi\circ\dot\gamma\right)'=0.
\end{equation}
Since $\Fc$ is positive-homogeneous of degree $1$, we have $\prc\Fc xi=\pprc\Fc ykxi y^k$. So it follows that relation \eqref{eq:Eul} is equivalent to
\[
\left(\pprc\Fc ykxi\circ\dot\gamma\right)(\gamma^k)'-\left(\pprc\Fc xkyi\circ\dot\gamma\right)(\gamma^k)'-\left(\pprc{\Fc} ykyi\circ\dot\gamma\right)(\gamma^k)''=0
\]
or (applying the second relation of \eqref{eq:crvcomp}) to
\begin{equation}\label{eq:Eul2}
  \left(\pprc\Fc ykxi y^k\right)\circ\dot\gamma-\left(\pprc\Fc xkyi y^k\right)\circ\dot\gamma-\left(\pprc{\Fc} ykyi\circ\dot\gamma\right)(\gamma^k)''=0.
\end{equation}

{\bf 2.} Now we give a simple proof of Hamel's following well-known theorem \cite{Ha}. For other approaches and proofs, see  \cite{C,R,S1,Szi}.

\begin{thma}
Let $M$ be a manifold which admits an atlas consisting of a single chart, and let $\Fc\colon TM\to M$ be a Finsler function. The Finsler space $(M,\Fc)$ is projectively flat if, and only if, there exists a coordinate system $(u^i)\en$ on $M$ with induced coordinate system $((x^i)\en,(y^i)\en)$ on $TM$ such that the equations

\begin{equation}\label{eq:Ham}
\pprc\Fc xkyi y^k=\pprc\Fc xiyk y^k
\end{equation}
are satisfied.
\end{thma}

\begin{proof}
\emph{Necessity.} Suppose that $(M,\Fc)$ is projectively flat, i.e., there exists a coordinate system $u=(u^i)\en$ on $M$ such that for every geodesic $\gamma\colon[t_1,t_2]\to M$ of $\Fc$ we have
\begin{equation}\label{eq:line}
u\circ\gamma(t)=f(t)a+b\mbox{ for all } t\in[t_1,t_2],
\end{equation}
where $f\colon[t_1,t_2]\to\RR$ is a $C^2$ function with positive derivative; $a,b \in\RR^n$, $a\neq0$ (these data depend, of course, on $\gamma$). Then
\[
\gamma^i(t)=f(t)a^i+b^i,\quad (\gamma^i)'(t)=f'(t)a^i,\quad(\gamma^i)''(t)=f''(t)a^i,
\]
therefore $(\gamma^i)''=\frac{f''}{f'}(\gamma^i)'$. Thus the last term on the left-hand side of \eqref{eq:Eul2} can be manipulated as follows:
\[
\left(\pprc{\Fc} ykyi\circ\dot\gamma\right)(\gamma^k)''=\frac{f''}{f'}\left(\pprc{\Fc} ykyi\circ\dot\gamma\right)(\gamma^k)'\overset{\eqref{eq:crvcomp}}= \frac{f''}{f'}\left(\pprc{\Fc} ykyi y^k\right)\circ\dot\gamma=0,
\]
taking into account the $0^+$-homogeneity of $\prc\Fc yi$. So \eqref{eq:Eul2} takes the simpler form
\[
\left(\pprc\Fc xkyi y^k\right)\circ\dot\gamma-\left(\pprc\Fc xiyk y^k\right)\circ\dot\gamma=0,
\]
which implies \eqref{eq:Ham}.
\medskip

\noindent\emph{Sufficiency.} If the Hamel equations are satisfied, then the Euler\,--\,Lagrange equations reduce substantially, and hence a $C^2$ curve $\gamma\colon[t_1,t_2]\to M$ is an extremal of $\mathbb F$ if, and only if,
\begin{equation}
\left(\pprc{\Fc} ykyi\circ\dot\gamma\right){\gamma^k}''=0.
\end{equation}
Thus, if we choose a point $p\in M$ and a vector $v\in T_pM\setminus\{0\}$, the curve $\gamma$ in $M$ given by
\[
u\circ\gamma(t)=t\,y(v)+u(p)
\]
is an extremal of $\mathbb F$ with initial velocity $\dot\gamma(0)=v$. If $\tilde\gamma$ is a reparameterization of $\gamma$  of constant speed $F(v)$, then $\tilde\gamma$ becomes a geodesic with initial velocity $v$, and still of the form \eqref{eq:line}.
  Since in a Finsler space there exists exactly one geodesic with given initial data, all geodesics are of the form \eqref{eq:line}, hence $(M,\Fc)$ is projectively flat.
\end{proof}

It follows that a coordinate system is \emph{geodesically linear}, or \emph{rectilinear}, if and only if, \eqref{eq:Ham} holds.

\bigskip

{\bf 3.} Now we investigate $1^+$-homogeneous Lagrange functions, whose level surfaces have non-vanishing Gaussian curvature.

\begin{thm}\label{thm:1}
Let $M$ be a manifold with an atlas consisting of a single chart. Let a positive-homogeneous Lagrange function $L$ on $TM$ be given, such that it is smooth and positive on $\mathring TM$ and its level surfaces have non-vanishing Gaussian curvature. Then $L$ is a Finsler function, and the Finsler manifold $(M,L)$ is projectively flat if, and only if, there exists a coordinate system $(u^i)\en$ on $M$ with induced coordinate system $((x^i)\en,(y^i)\en)$ on $TM$, such that
\begin{equation}
\frac{\pa^2 L}{\pa y^k\pa x^i}y^k=\frac{\pa^2L}{\pa y^i\pa x^k} y^k. \label{eq:12}
\end{equation}
\end{thm}

\bigskip
\begin{proof}
We show that under the given conditions $L$ is a Finsler function.

Let us consider a single tangent space $T_{p}M$ and a level surface of $L$ given by
$\Phi:=\{v\in T_pM\mid L(v)=1\}$. In a linear coordinate system $(y^i)\en$ on $T_pM$, for all $v\in\Phi$ we have $\frac{\pa L}{\pa y^i}y^i(v)=L(v)=1$, hence $L(v)$ is a regular value of $L$, and $\Phi$ is a hypersurface.

Choose an arbitrary point $v\in\Phi$. There is a basis $e_1,\dots,e_{n-1},e_n:=v$ of $T_pM$, with dual $(u^i)\en$, such that $\prc L e\alpha(v)=0$, $\alpha\in\{1,\dots,n-1\}$. Equip $T_pM$ with an Euclidean structure which makes this basis orthonormal. Set $V:=\operatorname{span}(e_1,\dots,e_{n-1})$. Then we can parametrize $\Phi$ on a neighbourhood of $v$ by the mapping
\[
\mathfrak r\colon U\subset V\to\Phi,\quad \mathfrak r:=u^\alpha e_\alpha+z\cdot v,
\]
where $z\colon U\to\RR$ is a smooth function (we agree that Greek indices run from $1$ to $n-1$).
Denoting the components of $\mathfrak r$ by $r^i:=u^i\circ\mathfrak r=u^i+z \delta^i_n$, the Gauss formula gives
\[
\pprc{r^i}u\alpha u\beta=\Gamma^\sigma_{\alpha\beta}\prc{r^i }u\sigma+h_{\alpha\beta} n^i,
\]
where $n^i$ and $h_{\alpha\beta}$ are the component functions of the unit normal vector field $\mathfrak n$ and of the second fundamental form, respectively, and $\Gamma^\sigma_{\alpha\beta}$ are the Christoffel symbols. The derivatives of $r^i$ are
\[
\prc{r^i}u\alpha=\delta^i_\alpha+\prc zu\alpha \delta^i_n,\quad \pprc{r^i}u\alpha u\beta=\pprc zu\alpha u\beta \delta^i_n.\tag{$\ast$}
\]
The mapping $\mathfrak r$ parametrizes $\Phi$, hence $L\circ\mathfrak r=1$.
Differentiating with respect to $u^\alpha$, we obtain
\[
\begin{aligned}[t]
0&=\prc{(L\circ\mathfrak r)}u\alpha=\left(\prc Lui\circ\mathfrak r\right)\prc{r^i}u\alpha=\left(\prc Lu\sigma\circ\mathfrak r\right)\prc{r^\sigma}u\alpha+\left(\prc Lu n\circ\mathfrak r\right)\prc{r^n}u\alpha\\
&\overset{(\ast)}=\prc Lu\alpha\circ\mathfrak r+\left(\prc Lu n\circ\mathfrak r\right)\prc zu\alpha.
\end{aligned}
\tag{$\ast\ast$}
\]
Since $\left(\prc Lui\circ\mathfrak r\right)(0)=\prc Lui(v)=\delta^n_i$, from $(\ast\ast)$ we have $\prc zu\alpha(0)=0$. As a result, the Gauss formula for $i=n$, evaluated at $0$, and taking into account $(\ast)$, gives
\[
\pprc zu\alpha u\beta(0)=h_{\alpha\beta}(0)n^n(0)=h_{\alpha\beta}(0).
\]
Now differentiate $(\ast\ast)$ with respect to $u^\beta$:
\[
0=\pprc Lu\beta u\alpha\circ\mathfrak r +\left(\pprc L u\beta un\circ\mathfrak r\right)\prc zu\alpha+\left(\prc Lun\circ\mathfrak r\right)\pprc zu\beta u\alpha.
\]
Evaluated at $0$, the second term vanishes, and we get
\[
\pprc Lu\beta u\alpha(v)=-\pprc zu\beta u\alpha(0)=-h_{\beta\alpha}(0).
\]
Since the Gauss curvature $K=\frac{\det(h_{\alpha\beta})}{\det(g_{\alpha\beta})}$ of $\Phi$ was supposed to be non-vanishing, $h_{\beta\alpha}$ has maximal rank $n-1$. So $\operatorname{rank}\left(\pprc Lu\beta u\alpha(v)\right)=n-1$.

If $(y^i)\en$ is an arbitrary linear coordinate system on $T_pM$, then $\prc{}ui=A^i_j\prc{}yj$ for some matrix $(A_i^j)$, hence
\[
\pprc Luk ul(v)=A^i_kA^j_l\pprc Lyiyj(v).
\]
Consequently, we also have that $\operatorname{rank}\left(\pprc Lyiyj(v)\right)=n-1$. The point $v$ was arbitrary in $\Phi$, so this holds everywhere on $\Phi$, and this is true for every $p\in M$. Thus the reduced Hessian of $L$ defines a non-degenerate symmetric bilinear form. Hence $L$ is a pseudo-Finsler function on $M$.

According to \cite[Lemma 2]{C}, a pseudo-Finsler function which is positive outside the zero vectors is a Finsler function. It follows that either $L$ or $-L$ is a Finsler function, and Theorem A yields that the geodesics of $L$ are straight lines if, and only if, \eqref{eq:12} holds.
\end{proof}
\bigskip

{\bf 4.} We know that a Riemannian space is of constant curvature if, and only if, it is projectively flat. Thus, as a corollary of Theorem A, we have the following.

\begin{cor}
A Riemannian space $(M,g)$ is of constant curvature if, and only if, there is an atlas $\mathcal A$ on $M$ such that for any chart $(U,(u^i)\en)$ in $\mathcal A$ with induced chart $(\tau^{-1}(U),((x^i)\en,(y^i)\en))$ we have
\[
\pprc{\sqrt{g_{ij}y^iy^j}} xkym y^k=\pprc{\sqrt{g_{ij}y^iy^j}} xmyk y^k.
\]
\end{cor}

\bigskip
Also we obtain a corollary for Randers spaces. Let $(M,\Fc)$ be a Randers space with metric function locally given by $\Fc=\sqrt{g_{ij}y^iy^j}+b_iy^i$, and suppose that $b_iy^i$ is a \emph{closed} one-form. In this case,
\begin{multline}\label{eq:Rand}
\pprc{\Fc}xsyi y^s-\pprc{\Fc}xiys y^s=\\
\begin{aligned}
&\ =\pprc{\sqrt{g_{kj}y^k y^j}}xsyi y^s-\pprc{\sqrt{g_{kj}y^k y^j}}xiys y^s-\left(\prc{b_s}xi-\prc{b_i}xs\right)y^s\\
&\overset{\textrm{cond.}}=\pprc{\sqrt{g_{kj}y^k y^j}}xsyi y^s-\pprc{\sqrt{g_{kj}y^k y^j}}xiys y^s
\end{aligned}
\end{multline}
By Theorem A, this means that the projective flatness of the Riemannian space $(M,g)$ is equivalent to the projective flatness of the Randers space $(M,\Fc)$. Thus we obtain

\begin{cor}
If in a Randers space the deforming $1$-form is closed, then it is (locally) projectively flat if, and only if, the corresponding Riemannian space is (locally) projectively flat.
\end{cor}

\bigskip
{\bf 5.}  Relation \eqref{eq:Ham} is not a tensor relation, as it is not preserved by changes of coordinate systems in general. So given a Finsler space $(M,\Fc)$ and a global coordinate system $(\bar u^i)\en$ on $M$, with induced coordinate system $((\bar x^i)\en,(\bar y^i)\en)$ on $TM$,
the fact that
\[
\frac{\pa^2 \Fc}{\pa \bar x{}^k \pa \bar y{}^i}\bar y{}^k = \frac{\pa^2 \Fc}{\pa \bar x^i\pa \bar y{}^k} \bar y{}^k
\]
is not true does not mean that $(M,\Fc)$ is not projectively flat, because there may exists another coordinate system $(u^i)\en$ such that
\eqref{eq:Ham} holds. We give conditions for the existence of such a coordinate system, and thus for the projective flatness of $(M,\Fc)$.

Let $(u^i)\en$ and $(\bar u^i)\en$ be two coordinate systems on $M$. We have the transformation rules
\begin{align*}
&\bar y^i=y^k\prc{\bar x^i}xk=y^k\left(\prc{\bar u^i}uk\right)^{\mathsf v} ,\\
&\prc{}xi=\prc{\bar x^k}xi\prc{}{\bar x}k+y^l\pprc{\bar x^k}xlxi\prc{}{\bar y}k=\left(\prc{\bar u^k}ui\right)^{\mathsf v}\prc{}{\bar x}k+y^l\left(\pprc{\bar u^k}ului\right)^{\mathsf v}\prc{}{\bar y}k,\\
&\prc{}yi=\prc{\bar x^k}xi\prc{}{\bar y}k=\left(\prc{\bar u^k}ui\right)^{\mathsf v}\prc{}{\bar y}k
\end{align*}
Where ${}^{\mathsf v}$ means the vertical lift.
Then
\[
\frac{\pa \Fc}{\pa y^i}=\frac{\pa \Fc}{\pa\bar y{}^r}\, \frac{\pa \bar x{}^r}{\pa x^i},
\]
and after calculating $\frac{\pa^2\Fc}{\pa x^s \pa y^i}$, we obtain
\begin{align}
& \frac{\pa^2\Fc}{\pa x^s \pa y^i}y^s -\frac{\pa^2\Fc}{\pa x^i\pa y^s} y^s= \label{eq:19} \tag{A} \\[5pt]
&\qquad\begin{aligned}&=\left(\dfrac{\pa^2 \Fc}{\pa \bar x{}^m \pa \bar y^l}-\dfrac{\pa^2\Fc}{\pa \bar x{}^l\pa \bar y{}^m}\right)
\left(\dfrac{\pa \bar u{}^m}{\pa u^s}\, \dfrac{\pa \bar u{}^l}{\pa u^i}\right)^{\mathsf v}y^s \\
&\quad+ \dfrac{\pa^2\Fc}{\pa \bar y{}^r \pa \bar y{}^l}\left[\dfrac{\pa^2\bar u{}^l}{\pa u^s\pa u^k}\, \dfrac{\pa\bar u{}^r}{\pa u{}^i}-
\dfrac{\pa^2\bar u{}^l}{\pa u^i\pa u^k}\, \dfrac{\pa\bar u{}^r}{\pa u^s}\right]^{\mathsf v} y^s y^k.
\end{aligned}\tag{B}
\label{eq:20a}
\end{align}
For a given Finsler function $\Fc$ and coordinate system $(u^i)\en$, $\eqref{eq:20a}=0$ is a second-order PDE system for the functions $(\bar u^i)\en$. If this PDE system is solvable, then \eqref{eq:19} vanishes, which means that $(M,\Fc)$ is projectively flat, and $(u^i)\en$ is a rectilinear coordinate system. Conversely, if $(M,\Fc)$ is projectively flat, then \eqref{eq:19} vanishes, and the functions $(\bar u{}^i)\en$ solve $\eqref{eq:20a}=0$. This yields
\begin{thm}\label{thm:2}
If $M$ is a manifold with a global coordinate system $(u^i)\en$ (and induced coordinate system $((x^i)\en,(y^i)\en)$ on $TM$), then a Finsler space $(M,\Fc)$ is projectively flat if, and only if, the second-order PDE system $\eqref{eq:20a}=0$ is solvable for $(\bar u^i)\en$.
\end{thm}

{\bf 6.} Let $(M,\Fc)$ be a Finsler manifold, $(u^i)\en$ a rectilinear coordinate system on $M$.
The question rises: which other coordinate systems $(\bar u^i)\en$ on $M$ are geodesically linear?

\begin{prop}\label{prop:1}
Let $(M,\Fc)$ be a Finsler manifold, $(u^i)\en$ a global rectilinear coordinate system on $M$. Then $\bar u=(\bar u^1,\dots,\bar u^n)$ is another rectilinear coordinate system if, and only if, the coordinate transformation $\bar u\circ u^{-1}$ is affine.
\end{prop}

\begin{proof}
Let $\varphi:=\bar u\circ u^{-1}\colon u(M)\to \bar u(M)$, $(e^i)\en$ the canonical coordinate system on $\RR^n$, and $D_i$ the usual partial derivative. If $\varphi$ is affine, then we have
\begin{equation}\label{eq:aff}
\pprc{\bar u^l}usuk=D_sD_k(\bar u^l\circ u^{-1})\circ u=D_sD_k(e^l\circ\varphi)\circ u=0,
\end{equation}
and hence the bracket in \eqref{eq:20a} vanishes. Then, taking into account the transformation rule of $\bar y^i$, from the equality of \eqref{eq:19} and \eqref{eq:20a} we obtain
\[
 0=\frac{\pa^2\Fc}{\pa x^s \pa y^i}y^s -\frac{\pa^2\Fc}{\pa x^i\pa y^s} y^s=
 \left(\dfrac{\pa^2 \Fc}{\pa \bar x{}^m \pa \bar y^l}-\dfrac{\pa^2\Fc}{\pa \bar x{}^l\pa \bar y{}^m}\right)\bar y^m
 \left(\dfrac{\pa \bar u{}^l}{\pa u^i}\right)^{\mathsf v}.
\]
This shows that $(\bar u^i)\en$ is a geodesically linear coordinate system.

Conversely, if both $(u^i)\en$ and $(\bar u^i)\en$ are geodesically linear, then the first parenthesis in \eqref{eq:20a} vanishes. Denoting the function in the bracket by $D^{lr}_{ski}$ we may write
\[
(D^{lr}_{ski})^{\mathsf v}\left[\frac{\pa^2\Fc}{\pa \bar y{}^l \pa \bar y{}^r}y^sy^k\right]=0.
\]
However, $\frac{\pa^2\Fc}{\pa \bar y{}^l \pa \bar y{}^r}y^sy^k$ is non-vanishing on every\ tangent space, thus $D_{ski}^{lr}$ must be constant zero. Let $A_k^l:=\frac{\pa \bar u^l}{\pa u^k}$, $(B_k^l):=(A_k^l)^{-1}$. Then the relation $D^{lr}_{ski}=0$ takes the form
\[
\left(\frac{\pa}{\pa u^s}A^l_k\right)A_i^r -\left(\frac{\pa}{\pa u^i}A_k^l\right) A_s^r=0.
\]
Multiplying by $B_r^m$, we obtain
\[
\left(\frac{\pa}{\pa u^s}A_k^l\right)\delta _i^m - \left(\frac{\pa}{\pa u^i} A_k^l\right) \delta_s^m =0.
\]
A contraction on $m$ and $i$ yields $0=\frac{\pa}{\pa u^s}A_k^l=\pprc{\bar u^l}usuk$, and hence we get \eqref{eq:aff} as desired.
\end{proof}

One can see in a similar way that the formula $\frac{\pa^2\Fc}{\pa x^k \pa y^i}$ is preserved by coordinate transformations if, and only if, the coordinate transformation is affine.

\bigskip
{\bf 7.} Let $(M,\Fc)$ be a Finsler space. We want to find a condition under which an infinitesimal transformation takes geodesics into geodesics as point sets. We represent an infinitesimal transformation by the velocity vector field $X$ of a flow $\varphi$ on $M$. We can assume that the flow is global, and $M$ admits a global coordinate system $(u^i)\en$. Consider a geodesic $\gamma\colon I\to M$, and define the curves
\[\gamma_s\colon t\in I\to \gamma_s(t):=\varphi(s,\gamma(t))\in M.
\]
The component functions of $\gamma_s$, $\dot\gamma_s$ and $\ddot\gamma_s$ are given by
  \begin{align}
  \gamma_s^k(t)&=\varphi^k(s,\gamma(t))\label{eq:gammas},\\
  (\gamma_s^k)'(t)&=\parcv{\varphi^k}{u^i}(s,\gamma(t))(\gamma^i)'(t)\label{eq:gammasdot},\\
  (\gamma_s^k)''(t)&=\parcv{\varphi^k}{u^i}(s,\gamma(t))(\gamma^i)''(t)+\parcv{^2\varphi^i}{u^i\partial u^j}(s,\gamma(t)) (\gamma^i)'(t)(\gamma^j)'(t),\label{eq:gammasddot}
  \end{align}
  for all $t\in I$, where $\varphi^i:=u^i\circ\varphi$. Since $X$ is the velocity vector field of $\varphi$, its component functions $X^i=y^i\circ X$ satisfy $X^i(p)=(s\mapsto \varphi^i(s,p))'(0)$. Hence, differentiating \eqref{eq:gammas}, \eqref{eq:gammasdot} and \eqref{eq:gammasddot} with respect to $s$, we obtain
  \begin{align*}
    (s\mapsto\gamma_s^k)'&=X^k\circ\gamma_s,\\
    (s\mapsto(\gamma_s^k)')'&=\left(\parcv{X^k}{u^i}\circ\gamma_s\right)(\gamma_s^i)',\\
    (s\mapsto(\gamma_s^k)'')'&=\left(\parcv{X^k}{u^i}\circ\gamma_s\right)(\gamma_s^i)''+\left(\parcv{^2X^i}{u^i\partial u^j}\circ\gamma_s\right)(\gamma_s^i)'(\gamma_s^j)'.
  \end{align*}
  If the curves $\gamma_s$ are geodesics as point sets, then
  \begin{equation}\label{eq:24}
  \parcv{\Fc}{x^i}\circ\dot\gamma_s=\left(\parcv{^2\Fc}{x^k\partial y^i}\circ\dot\gamma_s\right)(\gamma_s^k)'+\left(\parcv{^2\Fc}{y^k\partial y^i}\circ\dot\gamma_s\right)(\gamma_s^k)''
  \end{equation}
  holds. Differentiating this relation with respect to $s$ at $0$, a straightforward calculation gives
  \begin{equation}\label{eq:25}
  \begin{aligned}
&\left[\frac{\paa\Fc}{\pa x^s\pa x^i}-\frac{\pa^3\Fc}{\pa x^s\pa x^k\pa y^i}y^k-\frac{\pa^3\Fc}{\pa x^s\pa y^k\pa y^i}C^k\right] (X^s)^{\mathsf v}\\
&\qquad +\left[\left(\frac{\paa \Fc}{\pa y^s\pa x^i}-\frac{\paa \Fc}{\pa y^i\pa x^s}\right)y^r-\frac{\pa^2\Fc}{\pa y^s\pa y^i}C^{\,r} \right.\\
&\qquad -\left.\left(\frac{\pa^3\Fc}{\pa y^s\pa x^k\pa y^i}y^k
+\frac{\pa^3\Fc}{\pa y^s\pa y^k\pa y^i} C^k\right)y^r\right]\left(\frac{\pa X^s}{\pa u^r}\right)^{\mathsf v}\\
&\qquad -\left[\frac{\paa \Fc}{\pa y^s\pa y^i}y^my^r\right]\left(\frac{\paa X^s}{\pa u^m\pa u^r}\right)^{\mathsf v}=0
\end{aligned}
\end{equation}
evaluated at the tangent vectors $\dot\gamma(t)$, $t\in I$. Here, the symbols $C^k$ stand for some functions satisfying
\[
\frac{\paa \Fc}{\pa y^k\pa y^i}C^k=\frac{\pa \Fc}{\pa x^i}-\frac{\paa \Fc}{\pa x^k\pa y^i} y^k.
\]
Since $\gamma$ was an arbitrary geodesic, \eqref{eq:25} holds everywhere on $TM$.
So the component functions $X^i$ of the vector field $X$ must satisfy \eqref{eq:25} in order that $\varphi$ take geodesics into geodesics as point sets. That is \eqref{eq:25} is a necessary condition. However \eqref{eq:25} is also sufficient. Namely \eqref{eq:Eul2} and \eqref{eq:25} lead to \eqref{eq:24}. These yield

\begin{thm}\label{thm:3}
Given a Finsler space $(M,\Fc)$, the flow of a vector field $X$ on $M$ takes geodesics into geodesics as point sets if, and only if, in any coordinate-system the component functions $X^i$ of $X$ satisfy the second-order PDE system \eqref{eq:25}.
\end{thm}

In the PDE \eqref{eq:25}, $y^1,\dots,y^n$ can be considered as arbitrary parameters. However, if we replace the vertical lifts $(X^i)^{\mathsf v}$ with arbitrary functions $\xi^s$ on $TM$, and we complete \eqref{eq:25} with the equations
\[
\frac{\pa \xi^s}{\pa y^r}=0,  \label{eq:26}
\]
then we get a second order PDE system on $TM$ for the functions $\xi^s$.

If $\varphi$ is an infinitesimal isometry, then the role of \eqref{eq:25} is played by the Killing equations. Thus \eqref{eq:25} can be considered as a Killing type equation, when infinitesimal isometries are replaced by infinitesimal geodesic mappings.

\section{The parameter-preserving case}
In this section we investigate projectively flat Finsler spaces, such that in suitable coordinate systems the  geodesics are not just straight lines, but affinely parametrized ones. That is, any geodesic is of the form \mbox{$u(\gamma(t))=ta+b$} (see \eqref{eq:1b}). These Finsler spaces are called \emph{parameter preserving projectively flat}.

{\bf 1.} Let $(M,\Fc)$ be a Finsler space. It is easy to see that a curve $\gamma$ is a geodesic if, and only if, it satisfies the Euler\,--\,Lagrange equations for $\Fc^2$:
\begin{multline}
\frac{\partial \Fc^2}{\pa x^i}\circ\gamma -\frac {d}{dt}\left(\frac{\pa\Fc^2}{\pa y^i}\circ\gamma\right)=\\
=\frac{\pa\Fc^2}{\pa x^i}\circ\gamma-\left(\frac{\pa^2\Fc^2}{\pa x^k\pa y^i}\circ\gamma\right)(\gamma^k)'-\left(\frac{\pa^2\Fc^2}{\pa y^k\pa y^i}\circ\gamma\right)(y^k)''=0.
\label{eq:27}
\end{multline}
Namely, along any such geodesic $\gamma$ we obtain
\begin{multline*}
  \prc{\Fc^2}xi\mathop\circ\gamma-\frac d{dt}\left(\prc{\Fc^2}yi\circ\gamma\right)=\\=2\Fc\left[\prc\Fc xi\circ\gamma-\frac d{dt}\left(\prc\Fc yi\circ\gamma\right)\right]-2\left(\frac d{dt}(F\circ\gamma)\right)\prc\Fc yi\circ\gamma.
\end{multline*}
If $\gamma$ has constant speed, then $\Fc\circ\gamma$ is constant, and the last parenthesis drops out. Thus, for curves $\gamma$ with constant speed, relation \eqref{eq:27} and the vanishing of the bracket are equivalent.
Then the counterpart of Theorem A is the following.
\bigskip
\begin{thm}\label{thm:4}
On a manifold $M$ with an atlas consisting of a single chart, a Finsler space $(M,\Fc)$ is projectively flat in a parameter-preserving manner if, and only if, there exists a coordinate system $(u^i)\en$ on $M$ such that
\begin{equation}
\frac{\paa \Fcc}{\pa x^k\pa y^i}y^k=\frac 12\, \frac{\paa \Fcc}{\pa x^i\pa y^k}y^k. \label{eq:29}
\end{equation}
\end{thm}
\bigskip

\begin{proof}
If $(M,\Fc)$ is projectively flat in a parameter-preserving manner, then there exists a coordinate system $(u^i)\en$, such that any geodesics $\gamma$ is of the form $u\circ \gamma(t)=ta+b$ for some $a,b\in\RR^n$, $a\neq0$. Then the component functions of $\gamma$ are affine, that is $(\gamma^i)''=0$, and the same argument as in the proof of Theorem A yields
\[
\frac{\pa \Fcc}{\pa x^i}=\frac{\paa \Fcc}{\pa x^k\pa y^i}y^k.
\]
However, by the $2^+$-homogeneity of $\Fcc$, we have
\begin{equation}
\frac{\pa \Fcc}{\pa x^i}=\frac 12 \, \frac{\paa \Fcc}{\pa x^i\pa y^k}y^k, \label{eq:30} \end{equation}
which gives \eqref{eq:29}.

\smallskip
Conversely, if one has \eqref{eq:29}, then \eqref{eq:30} and \eqref{eq:27} yield
\begin{equation}
\left(\frac{\paa\Fcc}{\pa y^i\pa y^k}\circ\dot\gamma\right){\gamma^k}''=0.\label{eq:31}
\end{equation}
Since $\left(\frac{\paa\Fcc}{\pa y^i\pa y^k}\right)$ is non-degenerate, it follows that ${\gamma^k}''=0$, thus the component functions of $\gamma$ are of the form $\gamma^k(t)=ta^k+b^k$, hence $u\circ \gamma(t)=ta+b$.
\end{proof}

{\bf 2.} In our last theorem we consider a parameter preserving projectively flat Finsler space $(\RR^n,\Fc)$.
\begin{thm}\label{thm:5}
A Finsler space $(\RR^n,\Fc)$ is projectively flat in a parameter-preserving manner if, and only if, it is a Minkowski space.
\end{thm}

\begin{proof}
Let $(u^i)$ be the canonical coordinate system on $\RR^n$, and $((x^i),(y^i))$ the induced coordinate system on $T\RR^n$.
\medskip

\noindent\emph{Necessity.} Suppose that $(\RR^n,\Fc)$ is Minkowski space, that is, the metric function $\Fc\colon\RR^n\times\RR^n\to\RR$ is independent of its first $n$ variable.
Then $\prc\Fc xi=0$ and \eqref{eq:29} holds trivially (both sides vanish). Hence $(\RR^n,\Fc)$ is projectively flat in a parameter-preserving manner.
\medskip

\noindent\emph{Sufficiency.}
Let $\gamma\colon\RR\to\RR^n$ be a geodesic of $(\RR^n,\Fc)$. Geodesics have constant speed, hence the function $\Fc\circ\dot\gamma$ is constant, therefore
\begin{equation}\label{eq:const}
0=\frac d{dt}(\Fc\circ\dot\gamma)=\left(\prc\Fc xk\circ\dot\gamma\right){\gamma^k}'+\left(\prc\Fc yk\circ\dot\gamma\right){\gamma^k}''
\end{equation}
By condition, $\gamma(t)=ta+b$ ($t\in\RR$) for some $a,b\in\RR^n$, $a\neq0$, hence the functions ${\gamma^k}''$ vanish. Therefore \eqref{eq:const} reduces to $0=\left(\prc\Fc xk\circ\dot\gamma\right){\gamma^k}'$. Since $\dot\gamma$ can be any vector in $T\RR^n$, we can write $0=\prc\Fc xk y^k$. Differentiating this relation with respect to $y^i$, we obtain
\[
0=\pprc\Fc yixk y^k+\prc \Fc xi.
\]
Taking into account that $(\RR^n,\Fc)$ is projectively flat, and the homogeneity of $\Fc$, we have
\[
0=\pprc\Fc yixk y^k+\prc \Fc xi\overset{\eqref{eq:Ham}}=\pprc\Fc ykxi  y^k+\prc \Fc xi=2\prc\Fc xi,
\]
therefore $(\RR^n,\Fc)$ is Minkowskian.
\end{proof}

\begin{cor}
  A Finsler space $(\RR^n,\Fc)$ is a Minkowski space if and only if \eqref{eq:29} is satisfied.
\end{cor}

\begin{cor}
  A Riemann space $(\RR^n,g)$ is a Euclidean space if and only if
  \[
  \pprc{(g_{kj}y^k y^j)}xsyi y^s-\pprc{(g_{kj}y^k y^j)}xiys y^s.
  \]
\end{cor}

{\bf 3.} Does there exists a theorem similar to Theorem \ref{thm:5} such that $\Fc$ is continuous only? We give a positive answer. In both parts of the proof of Theorem \ref{thm:5} we used the Euler\,--\,Lagrange equations, which require $\Fc$ to be of class $C^2$, and so they can not be applied now.

To arrive at an appropriate generalization, consider a function ${}^*\Fc\colon\RR^n\times\RR^n\to\RR$ satisfying the following conditions:
\renewcommand\labelenumi{(\roman{enumi})}
\begin{enumerate}
  \item ${}^*\Fc$ is continuous, non-negative and ${}^*\Fc(p,v)$ is positive except if $v=0$.
  \item ${}^*\Fc$ is absolutely homogeneous in its second variable, i.e., ${}^*\Fc(p,\lambda v)=|\lambda|{}^*\Fc(p,v)$ for all $\lambda\in\RR$;
  \item ${}^*\Fc(p_,\cdot)$ is convex for all fixed $p_0\in\RR^n$;
  \item[(C)] along any parametrized straight line $g(t)=ta+b$ ($t\in\RR$) the limit
\[
  {}^*\Fc_g(v):=\lim_{t\to\infty}{}^*\Fc(g(t),v),
\]
exists, and the function ${}^*\Fc_g$ has the properties (i)--(iii).
\end{enumerate}
Then we say that ${}^*\Fc$ is a \emph{weak Finsler function}, and $(\RR^n,{}^*\Fc)$ is a \emph{weak Finsler space}. If
${}^*\Fc$ does not depend on its first argument, then condition (C) holds automatically, and $(\RR^n,{}^*\Fc)$ becomes a
Minkowski space in the classical sense (\cite{Tho}, p.\ 15). Keeping these in mind, we have the following generalization of Theorem \ref{thm:5}:

\begin{thm}\label{thm:6}
 $(\RR^n,{}^*\Fc)$ is a Minkowski space if, and only if, it is projectively flat in a parameter preserving manner.
\end{thm}

\begin{proof}
  \emph{Necessity.} Let $(\RR^n,{}^*\Fc)$ be a Minkowski space, and let $g(t)=ta+b$ be a parameterized straight line in $\RR^n$. We claim that $g$ is a geodesic. Let $p_1,p_2$ be points on $g$, and $q_1$ a further a point on $g$ such that
  \[
  \varrho(p_1,q_1)=\varrho(q_1,p_2)
  \]
  where $\varrho$ is the distance in $(\RR^n,{}^*\Fc)$. Let us observe that in a Minkowski space, the metric function ${}^*\Fc$ is translation invariant, i.e., ${}^*\Fc(p_1,v)={}^*\Fc(p_2,v)$. Therefore the symmetric geodesic spheres $C_{p_1}(r_1)$ and $C_{p_2}(r_1)$ centered at $p_1$ and $p_2$ resp., having radius $r_1=\varrho(p_1,q_1)$, are congruent (by translation), and they osculate at $q_1$:
  \[
  C_{p_1}(r_1)\cap C_{p_2}(r_1)=q_1.
  \]
  Therefore
  \[
  \varrho(p_1,q_1)+\varrho(q_1,p_2)=\varrho(p_1,p_2).
  \]
  This means that the point $q_1$ on $g$ is a point of the extremal joining $p_1$ and $p_2$. Repeating this process for the segments $p_1,q_1$ and $q_1,p_2$ we obtain further two points $q_2,q_3$ on $g$, which are also points of the extremal. Continuing this process we obtain an everywhere dense set of points in the extremal between $p_1$ and $p_2$. Thus the segment of $g$ between its two arbitrary points is an extremal. However, $g$ has constant velocity vector $g'(t)=a$, thus it also has constant speed, because ${}^*\Fc$ is translation invariant. Therefore $g$ is a geodesic of $(\RR^n,{}^*\Fc)$. Since $g$ was an arbitrary straight line in $\RR^n$, we obtain that the given Minkowski space is projectively flat in a parameter preserving manner.
  \medskip

  \noindent\emph{Sufficiency.} Suppose that $(\RR^n,{}^*\Fc)$ is projectively flat in a parameter-preserving manner. Then the geodesics are of the form $g(t)=ta+b$ ($t\in\RR$), where $a\in\RR^n$ is the constant velocity vector of the geodesic. The vector $a$ depends on $g$ alone, and we denote it by $a(g)$.

  Define the indicatrix of ${}^*\Fc$ at a point $p\in\RR^n$ as the set ${}^*\Ic(p):=\{v\in\RR^n\mid {}^*\Fc(p,v)=1\}$. In the remainder of the proof, by a straight line we mean a geodesic with constant speed $1$, more precisely a curve $g\colon\RR\to\RR^n$ of the form $g(t)=ta+b$ satisfying ${}^*\Fc(g(t),g'(t))=1$. It follows that for any straight line $g$ and for any of its points $p$, we have $a(g)\in {}^*\Ic(p)$.

  Let $p_0$ be an arbitrary point of $\RR^n$, and $g_0$ a straight line through it. Let $q$ be an arbitrary point of $\RR^n$ outside the image of $g_0$, and $g^*$ a straight line through $q$ parallel to $g_0$. Finally let $\ell$ be a straight line through $q$ intersecting $g_0$ in $p$.
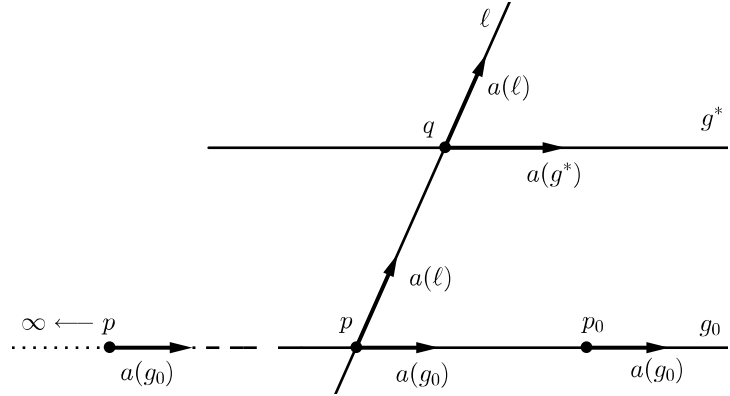
  \begin{figure}[!h]
  \begin{center}
  \includegraphics{abra14pt}
  \end{center}
  \end{figure}

  \noindent Let us turn $\ell$ around $q$ until it coincides with $g^*$ (as a point set). Then $p$ tends to infinity along $g_0$. Denoting these `convergences' by the arrow $\to$, we have $a(\ell)\to a(g^*)$, and by (C), also $a(\ell)\to a(g_0)$. Consequently $a(g_0)=a(g^*)$. Since $a(g_0)\in {}^*\Ic(p_0)$ and $a(g^*)={}^*\Ic(q)$, and the direction of $g_0$ can be arbitrary, we obtain that the indicatrices ${}^*\Ic(p_0)$ and ${}^*\Ic(q)$ coincide (by translation). The point $q$ was also arbitrary, hence all the indicatrices are the same, which means that $(\RR^n,{}^*\Fc)$ is Minkowskian.
\end{proof}

\end{document}